\documentclass{amsart}
\usepackage{latexsym}
\usepackage{amssymb,amsmath}
\usepackage{geometry}
\usepackage{multirow}
\newlength\tindent
\usepackage[backend=biber,
style=numeric,
sorting=ynt
]{biblatex}
\usepackage{tikz}
\usetikzlibrary{chains}
\tikzset{node distance=2em, ch/.style={circle,draw,on chain,inner sep=2pt},chj/.style={ch,join},every path/.style={shorten >=4pt,shorten <=4pt},line width=1pt,baseline=-1ex}

\let\dlabel=\alabel

\newcommand{\dnode}[2][chj]{%
\node[#1,label={below:\dlabel{#2}}] {};
}

\newcommand{\dnodenj}[1]{%
\dnode[ch]{#1}
}

\newcommand{\dnodebr}[1]{%
\node[chj,label={below right:\dlabel{#1}}] {};
}

\newcommand{\dydots}{%
\node[chj,draw=none,inner sep=1pt] {\dots};
}

\newtheorem{theorem}{Theorem}
\newtheorem{lemma}{Lemma}

\newtheorem*{conjecture}{Conjecture}

\title{Cosmall Roots and Curve Neighborhoods}
\author{Chi-Nuo Lee and Arthur Wang}
\date{\today}

\begin{document}

\maketitle

\section{Introduction}

Let $X = G/P$ be a homogeneous space defined by a connected, simply connected, semisimple complex Lie group $G$ and a parabolic subgroup $P$. Given a subvariety $\Omega \subset X$, and a degree $d \in H_2(X)$, we define the \emph{curve neighborhood} $\Gamma_d(\Omega)$ to be the closure of the union of all rational curves of degree $d$ in $X$ that meet $\Omega$. Curve neighborhoods are objects of interest since they have applications in the study of quantum cohomology and quantum K-theory. 

The study of $\Gamma_d(\Omega)$ in \cite{BCMP} by Buch, Chaput, Mihalcea, and Perrin led to the result that if $\Omega$ is a Schubert variety in $X$, then so is $\Gamma_d(\Omega)$. This study was continued in \cite{BM} where an explicit combinatorial formula was found for the Weyl group element corresponding to $\Gamma_d(\Omega)$ when $\Omega \subset X$ is a Schubert variety. 

Fix a maximal torus $T$ and a Borel subgroup $B$ such that $T \subset B \subset P \subset G$. Let $R$ be the associated root system, with positive roots $R^+$ and simple roots $\Delta \subset R^+$. Let $W$ be the Weyl group of $G$ and $W_P$ the Weyl group of $P$. Given a $w \in W$, we define the Schubert variety $X(w) = \overline{Bw.P}$ in $X$. If $w$ is the minimal representative for its coset in $W/W_P$, then we have $\dim(X(w)) = l(w)$, where the length of a Weyl group element is the minimal number of simple reflections it can be the product of.  The parabolic subgroup $P$ corresponds to the set of simple roots $\Delta_P = \{\beta \in \Delta \mid s_{\beta} \in W_P\}$. The group $W_P$ is generated by $s_{\beta}$ for $\beta \in \Delta_P$. Set $R_P = R \cap \mathbb{Z}\Delta_P$ and $R^{+}_P = R^+ \cap \mathbb{Z}\Delta_P$, where $\mathbb{Z}\Delta_P =$ Span$_{\mathbb{Z}}(\Delta_P)$ is the group spanned by $\Delta_P$. Each root $\alpha \in R$ has a coroot $\alpha^{\vee} = \frac{2\alpha}{(\alpha,\alpha)}$.

Given a positive root $\alpha$ with $s_{\alpha} \notin W_P$, let $C_{\alpha} \subset X$ be the unique $T$-stable curve that contains the $T$-fixed points $1.P$ and $s_{\alpha}.P$. The homology group $H_2(X) = H_2(X;\mathbb{Z})$ can be identified with the quotient $\mathbb{Z}\Delta^{\vee}/\mathbb{Z}\Delta^{\vee}_P$, where $\mathbb{Z}\Delta^{\vee}$ is the coroot lattice of $G$ and $\mathbb{Z}\Delta^{\vee}_P$ is the coroot lattice of $P$. Under this identification, the degree $[C_{\alpha}] \in H_2(X)$ is equal to the image of the coroot $\alpha^{\vee}$. 

We can describe $\Gamma_d(X(w))$ using the \emph{Hecke product} on $W$. For $u \in W$ and $\beta \in \Delta$, define 

\[
u \cdot s_{\beta} =
     \begin{cases} 
      us_{\beta} & \text{if } l(us_{\beta}) > l(u)\\
      u & \text{if } l(us_{\beta}) < l(u)
   \end{cases}
\] 
Let $u$,$v \in W$ and let $v = s_{\beta_1}s_{\beta_2}\cdots s_{\beta_l}$ be any reduced expression for $v$. Define the \emph{Hecke product} of $u$ and $v$ 

\begin{equation*}
    u \cdot v = u \cdot s_{\beta_1} \cdot s_{\beta_2} \cdot \ldots \cdot s_{\beta_l},
\end{equation*} 
where simple reflections are multiplied to $u$ in a left to right order.

Given a degree $d \in H_2(X) = \mathbb{Z}\Delta^{\vee}/\mathbb{Z}\Delta^{\vee}_P$, the maximal elements of the set $\{\alpha \in R^+ \setminus R^{+}_P \mid \alpha^{\vee} + \mathbb{Z}\Delta^{\vee}_P \leq d\}$ are called \emph{maximal roots} of $d$. Then we have the following theorem from \cite{BM}, which is formulated using the \emph{Hecke product},

\begin{theorem}
Let $d$ be a non-zero degree and $\alpha$ be a maximal root of $d$. For any $w \in W$ we have $\Gamma_d(X(w)) = \Gamma_{d-\alpha^{\vee}}(X(w \cdot s_{\alpha}))$.
\end{theorem} 

Theorem 1 relates the curve neighborhood of any Schubert variety to the curve neighborhood of a point. In fact, if we define $z^{P}_d \in W^{P}$ by $\Gamma_d(1.P) = X(z^{P}_d)$, then Theorem 1 implies that for any $u \in W$, $\Gamma_{d}(X(u)) = X(u \cdot z^{P}_{d})$. Much of the work done in \cite{BM} focuses on the curve neighborhoods of a point. 

A root $\alpha \in R^+ \setminus R^{+}_P$ is called $P$-cosmall if $\alpha$ is a maximal root of $\alpha^{\vee} + \Delta^{\vee}_P \in H_2(X;\mathbb{Z})$. The highest root in $R$ is always $P$-cosmall for every parabolic subgroup $P$. If we consider the special case of $P = B$, then we can talk about $B$-cosmall roots, which we also refer to as cosmall roots. Note that any $P$-cosmall root is cosmall and if $R$ is simply laced, then all roots are cosmall. In general if $\alpha$ is a simple root or a long root, then $\alpha$ is cosmall. We have the following conjecture from \cite{BM} concerning $P$-cosmall roots.

\begin{conjecture}
Assume that $R$ is simply laced and let $\alpha \in R^+\setminus R^{+}_P$. 
Then $\alpha$ is $P$-cosmall if and only if $\Gamma_{\alpha^\vee}(1.P) = X(s_{\alpha})$.
\end{conjecture}

The conjecture has been proven in its entirety in the case that $P = B$ in \cite{BM}. Note that the result for the forward direction is immediate through an application of Theorem 1. The full proof of the conjecture is given in Section 2, which relies on the following theorem. 

First we define a \emph{greedy decomposition} of $d \in H_2(X)$ to be a sequence of positive roots $(\alpha_1,\alpha_2,\dots,\alpha_k)$ such that $\alpha_1 \in  R^+\setminus R^{+}_P$ is a maximal root of $d$ and $(\alpha_2,\dots,\alpha_k)$ is a greedy decomposition of $d-\alpha^{\vee}_1 \in H_2(X)$. Note that the greedy decomposition of $d$ is unique up to a reordering \cite{BM}. Suppose $(\alpha_1,\alpha_2,\dots,\alpha_k)$ is a greedy decomposition of $d$. Then the coset $z^{P}_{d}W_P$ in $W/W_P$ is defined as 

\begin{equation} \label{eq1}
    z^{P}_dW_P = s_{\alpha_1} \cdot s_{\alpha_2} \cdot \ldots \cdot s_{\alpha_k}W_P
\end{equation}
where multiplication is defined by the \emph{Hecke product}.

\begin{theorem}
The root $\alpha \in R^+ \setminus R^+_P$ is $P$-cosmall if and only if $\Gamma_{\alpha^{\vee}}(1.P) = X(s_{\alpha})$ and the greedy decomposition of $d$ has length 1.
\end{theorem}

Next we provide alternative conditions for a root to be $P$-cosmall by considering the set $\Delta(\alpha) = \{\beta \in \Delta \mid \alpha + \beta \in R\}$ along with the set $\Delta_P$. This approach leads to more explicit criterion for a root to be $P$-cosmall. We state this result in our third and final theorem, which is a type independent result.

\begin{theorem}
Let $\alpha \in R^+ \setminus R^+_P$ be a cosmall root. Then $\alpha$ is $P$-cosmall if and only if $\Delta(\alpha) \cap \Delta_P = \emptyset$.
\end{theorem}

Our paper is organized in the following manner. In section 2, we prove Theorem 2 and show that it implies the conjecture. Theorem 3 is proven in section 3. In section 4, we provide tables that compute $\Delta(\alpha)$ for cosmall roots in the case of classical root systems. This makes it easy to check if a root is $P$-cosmall. 

This paper is a product of the DIMACS REU program and we are grateful to the DIMACS REU and the Rutgers Mathematics Department for their funding and support. Funding was also provided by the NSF through the grant DMS-1503662. We would like to thank our advisor Professor Anders Buch for introducing us to this subject as well as his PhD student Sjuvon Chung for his helpful discussions. We would like to mention Professor Buch's maple program Equivariant Schubert Calculator \cite{EC}, which allowed us to come up with and verify some of our initial results and undoubtedly will be useful for further study in this subject. During the course of writing this paper, we became aware that the conjecture from \cite{BM} has been proven in Christoph B{\"a}rligea's thesis \cite{CB}, along with several other interesting results about $P$-cosmall roots. Our Theorem 3 is not explicitly stated in his thesis though it can possibly be derived using his work on $P$-indecomposable roots.

\bigskip

\section{Classification of P-cosmall roots}

In this section, we prove Theorem 2 and derive the conjecture from \cite{BM} as a consequence. For this section, we fix a positive root $\alpha \in R^{+}\setminus R^{+}_P$ and set $d = \alpha^{\vee} + \mathbb{Z}\Delta^{\vee}_P \in H_2(X;\mathbb{Z})$. Our goal is to prove the previously stated conjecture. We now present a proof of the conjecture in its fullest generality. We require the following lemma which is proven in \cite{FW}.

\begin{lemma}
Let $\alpha  \in R^+\setminus R^{+}_P$. Then $\alpha$ is uniquely determined by the coset $s_{\alpha}W_P \in W/W_P$.
\end{lemma}

\begin{proof}[Proof of Theorem 2]
Note that the condition $\Gamma_{\alpha^{\vee}}(1.P) = X(s_{\alpha})$ is equivalent to saying $s_{\alpha}W_P = z^{P}_{d} W_P$. Suppose that $\alpha$ is $P$-cosmall. By definition of $P$-cosmall, $(\alpha)$ is a greedy decomposition of $d$ so $s_{\alpha}W_P = z^{P}_{d}W_P$. Now suppose that $s_{\alpha}W_P = z^{P}_dW_P$ and the greedy decomposition of $d$ has length 1. If the greedy decomposition of $d$ is $(\gamma)$, then  $s_{\alpha}W_P = z^{P}_dW_P = s_{\gamma}W_P$. However by Lemma 1, we must have that $\alpha = \gamma$. Therefore it follows that $\alpha$ is $P$-cosmall.
\end{proof}

The Conjecture from \cite{BM} follows from Theorem 2 together with the following Lemma. Note that Theorem 2 holds in all Lie types. 

\begin{lemma}
Assume that $R$ is simply laced. Then the greedy decomposition of $d$ has length 1. 
\end{lemma}

\begin{proof}
Let $\gamma$ be any maximal root of $d$ satisfying $\gamma \geq \alpha$. Since $R$ is simply laced, it follows that $\gamma^{\vee} \geq \alpha^{\vee}$. Thus, we have that $\gamma^{\vee} + \mathbb{Z}\Delta^{\vee}_P = d$, so by definition and uniqueness, the greedy decomposition of $d$ is just $(\gamma)$.
\end{proof}

\section{Combinatorial Characterization of $P$-Cosmall roots}

We now provide an explicit combinatorial description of the $P$-cosmall roots of a root system $R$. Since $P$-cosmall roots are cosmall, we formulate our criteria for $P$-cosmall roots in terms of cosmall roots and the set $\Delta_P$. Given any positive root $\alpha$ define the set $\Delta(\alpha) = \{\beta \in \Delta | \alpha + \beta \in R\}$. Our proof of Theorem 3 relies on the following two lemmas.

\begin{lemma}
Let $\alpha,\beta \in R$ be arbitrary roots. Choose the maximal $k \in \mathbb{Z}$ such that $\alpha + k\beta \in R$. Then the length of $\alpha + k\beta$ is greater than or equal to the length of $\alpha$. 
\end{lemma}

\begin{proof}
Let $k,l \in \mathbb{Z}$ be maximal nonnegative integers such that $\alpha + k\beta, \alpha - l\beta \in R$, i.e. $\alpha + k\beta$ and $\alpha - l\beta$ are the ends of the $\beta$-string through $\alpha$. Let $\sigma_{\beta}$ be the reflection that sends $\beta \mapsto -\beta$ . It was shown in \cite{HUM} that $\sigma_{\beta}(\alpha+ k\beta) =\alpha - l\beta$. Therefore the vector $\alpha + \frac{k-l}{2}\beta$ is perpendicular to $\beta$. By the Pythagorean Theorem, we have $|\alpha|^2=|\alpha+ \frac{k-l}{2}\beta|^2+|\frac{k-l}{2}\beta|^2 \leq |\alpha+ \frac{k-l}{2}\beta|^2+|\frac{k+l}{2}\beta|^2 = |\alpha+ k\beta|^2$. 
\end{proof}

\begin{lemma}
Let $\alpha \in R^+$ be any short cosmall root and take $\beta \in \Delta(\alpha)$. Then $\beta \nleq \alpha$. 
\end{lemma}

\begin{proof}
Note that this statement is clear if $R$ is simply laced since all roots are long by convention. If $R$ is a classical root system of either type $B$ or $C$, then this statement can be verified easily through the tables provided in section 4. The lemma has been checked by computer for $F_4$ and $G_2$.
\end{proof}

\begin{proof}[Proof of Theorem 3]
Assume that $\alpha$ is $P$-cosmall and $\Delta(\alpha) \cap \Delta_P \neq \emptyset$. Choose $\beta \in \Delta(\alpha) \cap \Delta_P$. Choose $k \in \mathbb{N}$ maximal such that $\alpha + k\beta \in R^+$. Let $\gamma = \alpha + k\beta$. Then by Lemma 3, we have that $\lvert \gamma \rvert \geq \lvert \alpha \rvert$. This means that $\lvert \gamma^{\vee} \rvert \leq \lvert \alpha^{\vee} \rvert$. Using this fact along with $\beta^{\vee} \in \Delta^{\vee}_P$, we can conclude that $\gamma^{\vee} + \mathbb{Z}\Delta^{\vee}_P \leq \alpha^{\vee} + \mathbb{Z}\Delta^{\vee}_P$, which contradicts that $\alpha$ is $P$-cosmall.  

\bigskip

Assume now that $\Delta(\alpha) \cap \Delta_P = \emptyset$ and $\alpha$ is not $P$-cosmall. Then choose $\gamma \in R^+$ such that $\gamma > \alpha$ and $\gamma^{\vee} + \mathbb{Z}\Delta^{\vee}_P \leq \alpha^{\vee} + \mathbb{Z}\Delta^{\vee}_P$. Choose $\beta \in \Delta(\alpha)$ such that $\alpha + \beta \leq \gamma$. By assumption, we must have that $\beta \notin \Delta_P$. Now by considering the coefficient of $\beta^{\vee}$ in the inequality $\gamma^{\vee} + \mathbb{Z}\Delta^{\vee}_P \leq \alpha^{\vee} + \mathbb{Z}\Delta^{\vee}_P$, we can conclude that $\beta \leq \alpha$ and $\lvert \alpha \rvert < \lvert \gamma \rvert$. The second fact implies that that $\alpha$ is a short cosmall root, but this contradicts the statement of Lemma 4. 
\end{proof}

\section{Tables of $P$-cosmall roots in classical root systems}

Here we expand on the table provided in \cite{BM} by listing $\Delta(\alpha)$ for each cosmall root in the classical root systems. This allows one to easily check if a root is $P$-cosmall since it suffices to check if $\Delta(\alpha) \cap \Delta_P = \emptyset$. Since $R$ will either be of the form $A_{l-1},B_l,C_l$, or $D_l$, we can identify $R$ with a subset of $\mathbb{R}^l$. Let $e_1, \ldots, e_l$ be the standard basis for $\mathbb{R}^l$, and set $\beta_i = e_{i} - e_{i+1}$ for $1 \leq i \leq l - 1$. For type B, C, and D, define $\beta_l$ to be $e_l$, $2e_l$, and $e_{l-1} + e_{l}$ respectively. Recall that in case of types $A$ and $D$, all roots are cosmall. Note that $\Delta(\alpha)$ is a set of simple roots and if the index of a simple root $\beta_i$ is out of range, it should be interpreted as not being a member of the set. 

\bigskip

\subsection*{Type A}

Assume that $R$ is a root system of type $A$ with Dynkin diagram: 

\begin{align*}
\begin{tikzpicture}[start chain]
\dnode{1}
\dnode{2}
\dydots
\dnode{l-2}
\dnode{l-1}
\end{tikzpicture}
\end{align*}

\begin{center}
\begin{tabular}{|l|l|l|l|}
\hline
Simple & \multicolumn{3}{|l|}{$\beta_1,\beta_2,\dots,\beta_{l-1}$}           
\\ \hline
Cosmall & $e_i - e_j = \beta_i + \cdots + \beta_{j-1}$ & $1 \leq i < j \leq l$ & $\Delta(\alpha) = \{\beta_{i-1}, \beta_j\}$ 
\\ \hline
\end{tabular}
\end{center}

\bigskip

\subsection*{Type B}

Assume that $R$ is a root system of type $B$ with Dynkin diagram:

\begin{align*}
\begin{tikzpicture}[start chain]
\dnode{1}
\dnode{2}
\dydots
\dnode{l-1}
\dnodenj{l}
\path (chain-4) -- node[anchor=mid] {\(\Rightarrow\)} (chain-5);
\end{tikzpicture}
\end{align*}

\begin{center}
\begin{tabular}{|l|l|l|l|}
\hline
Simple & \multicolumn{3}{|l|}{$\beta_1,\beta_2,\dots,\beta_{l-1},\beta_l$}                       
\\ \hline
\multirow{2}{*}{Long} & $e_i - e_j = \beta_i + \cdots + \beta_{j-1}$ & \multicolumn{2}{|l|}{$1 \leq i < j \leq l$}
\\ \cline{2-4} & $e_i + e_j = \beta_i + \cdots + \beta_{j-1} + 2\beta_j + \cdots + 2\beta_l$ & \multicolumn{2}{|l|}{$1 \leq i < j \leq l$}
\\ \hline
Short & $e_i = \beta_i + \cdots + \beta_l$ & \multicolumn{2}{|l|}{$1 \leq i \leq l$} 
\\ \hline
\multirow{3}{*}{Cosmall}  & \multicolumn{2}{|l|}{$e_l = \beta_l$} & $\Delta(\alpha) = \{\beta_{l-1}\}$
\\ \cline{2-4} & $e_i - e_j = \beta_i + \cdots + \beta_{j-1}$ & $1 \leq i < j \leq l$ & $\Delta(\alpha) = \{\beta_{i-1},\beta_{j}\}$
\\ \cline{2-4} & $e_i + e_j = \beta_i + \cdots + \beta_{j-1} + 2\beta_j + \cdots + 2\beta_l$ & $1 \leq i < j \leq l$ & $\Delta(\alpha) = \{\beta_{i-1}, \beta_{j-1}\} \setminus {\beta_i}$   
\\ \hline
\end{tabular}
\end{center}

\bigskip

\subsection*{Type C}

Assume that $R$ is a root system of type $C$ with Dynkin diagram:

\begin{align*}
\begin{tikzpicture}[start chain]
\dnode{1}
\dnode{2}
\dydots
\dnode{l-1}
\dnodenj{l}
\path (chain-4) -- node[anchor=mid] {\(\Leftarrow\)} (chain-5);
\end{tikzpicture}
\end{align*}

\begin{center}
\begin{tabular}{|l|l|l|l|}
\hline
Simple & \multicolumn{3}{|l|}{$\beta_1,\beta_2,\dots,\beta_{l-1},\beta_l$}     
\\ \hline
Long & $2e_i = 2\beta_i + \cdots + 2\beta_{l-1} + \beta_l$ & \multicolumn{2}{|l|}{$1 \leq i \leq l$}
\\ \hline
\multirow{2}{*}{Short} & $e_i - e_j = \beta_i + \cdots + \beta_{j-1}$ & \multicolumn{2}{|l|}{$1 \leq i< j \leq l$}
\\ \cline{2-4} & $e_i + e_j = \beta_i + \cdots + \beta_{j-1} + 2\beta_j + \cdots + 2\beta_{l-1} + \beta_l$ & \multicolumn{2}{|l|}{$1 \leq i < j \leq l$} 
\\ \hline
\multirow{2}{*}{Cosmall} & $2e_i = 2\beta_i + \cdots + 2\beta_{l-1} +\beta_l$ & $1 \leq i \leq l$ & $\Delta(\alpha) = \{\beta_{i-1}\}$
\\ \cline{2-4} & $e_i - e_j = \beta_i + \cdots \beta_{j-1}$ & $1 \leq i < j \leq l$ & $\Delta(\alpha) = \{\beta_{i-1},\beta_j\}$
\\ \hline
\end{tabular}
\end{center}

\bigskip

\subsection*{Type D}

Assume that $R$ is a root system of type $D$ with Dynkin diagram:

\begin{align*}
\begin{tikzpicture}
\begin{scope}[start chain]
\dnode{1}
\dnode{2}
\node[chj,draw=none] {\dots};
\dnode{l-2}
\dnode{l-1}
\end{scope}
\begin{scope}[start chain=br going above]
\chainin(chain-4);
\dnodebr{l}
\end{scope}
\end{tikzpicture}
\end{align*}

\begin{center}
\begin{tabular}{|l|l|l|l|}
\hline
Simple & \multicolumn{3}{|l|}{$\beta_1,\beta_2,\dots,\beta_{l-1}, \beta_l$} 
\\ \hline
\multirow{7}{*}{Cosmall} & \multirow{2}{*}{$e_i -e_j = \beta_i + \cdots + \beta_{j-1}$} & \multirow{2}{*}{$1 \leq i < j \leq l$} & $\Delta(\alpha) = \{\beta_{i-1},\beta_j\}$
\\ & & & if $i \neq l-1$ or $j \neq l$
\\ \cline{2-4} & \multirow{2}{*}{$e_i -e_j = \beta_i + \cdots + \beta_{j-1}$} & \multirow{2}{*}{$1 \leq i < j \leq l$} & $\Delta(\alpha) = \{\beta_{i-1}\}$
\\ & & & if $i = l-1$ and $j = l$
\\ \cline{2-4} & $e_i + e_l = \beta_i + \cdots + \beta_{l-2} + \beta_l$ & $1 \leq i \leq l-1$ & $\Delta(\alpha) = \{\beta_{i-1},\beta_{l-1}\}$
\\ \cline{2-4} & $e_i + e_j = \beta_i + \cdots + \beta_{j-1} + 2\beta_j + $ & \multirow{2}{*}{$1 \leq i < j \leq l-1$} & \multirow{2}{*}{$\Delta(\alpha) = \{\beta_{i-1}, \beta_{j-1}\} \setminus \{\beta_i\}$} 
\\ & \hspace*{1cm} $\cdots + 2\beta_{l-2} + \beta_{l-1} + \beta_l$ & & 
\\ \hline
\end{tabular}
\end{center}

\emergencystretch=2em
\printbibliography
\end{document}